\documentclass[12pt]{amsart}
\usepackage[all]{xy}
\usepackage{amssymb}
\usepackage{color}
\calclayout
\newtheorem{dummy}{anything}[section]
\newtheorem{theorem}[dummy]{Theorem}

\newtheorem{lemma}[dummy]{Lemma}

\newtheorem{proposition}[dummy]{Proposition}

\newtheorem{corollary}[dummy]{Corollary}

\theoremstyle{definition}
\newtheorem{definition}[dummy]{Definition}

\newtheorem{example}[dummy]{Example}

\newtheorem{remark}[dummy]{Remark}

\newtheorem{problem}[dummy]{Problem}

\newcommand
{\eqncount}{\setcounter{equation}{\value{dummy}}%
\addtocounter{dummy}{1}}

\newcommand{\bZ}{\mathbb Z}

\newcommand{\La}{\Lambda}

\newcommand{\Ga}{\Gamma}

\newcommand{\mmatrix}[4]{\left (\vcenter
{\xymatrix@C-2pc@R-2pc{#1&#2\\#3&#4} } \right )}

\DeclareMathOperator{\Hom}{Hom}                                                                     
                                                                    
                                                                \DeclareMathOperator{\im}{Im}
                                                                    \DeclareMathOperator{\Ker}{Ker}
                                                                   \DeclareMathOperator{\pt}{pt}
                                                             \DeclareMathOperator{\Kernel}{Kernel}

\DeclareMathOperator{\hepta}{aut}
   
\DeclareMathOperator{\id}{Id}                                                                              \DeclareMathOperator{\coker}{coker}

                                                                           \DeclareMathOperator{\Rad}{Rad}

\begin{document}
 \title[Homotopy classification of $PD_4$-complexes relative an order relation]
{Homotopy classification of $PD_4$-complexes relative an order relation}
\author{Friedrich Hegenbarth, Mehmetc\.{i}k Pamuk and Du\v{s}an Repov\v{s}}
\address{Department of Mathematics
\newline\indent
University of Milano
\newline\indent
Milano, Italy} \email{friedrich.hegenbarth{@}unimi.it} 

\address{Department of Mathematics
\newline\indent
Middle East Technical University
\newline\indent
Ankara 06531, Turkey} \email{mpamuk{@}metu.edu.tr}

\address{Faculty of Education and Faculty of Mathematics and Physics
\newline\indent
University of Ljubljana
\newline\indent
Ljubljana 1000, Slovenia} \email{dusan.repovs{@}guest.arnes.si}

\date{\today}
\begin{abstract}
\noindent
We define an order relation among oriented $PD_4$-complexes.  We show that with respect to this 
relation, two $PD_4$-complexes over the same complex are homotopy equivalent if and only if there 
is an isometry between the second homology groups.  We also consider minimal objects of this relation.
\end{abstract}

\maketitle
\section{INTRODUCTION}

Let $X$, $P$ be  compact oriented $PD_4$-complexes and  $[X]\in H_4(X; \bZ)$,  $[P]\in H_4(P; \bZ) $ be their
fundamental classes, respectively.  We are going to use the notation $X \succ P$ if there is a continuous map
$f\colon X \to P$ such that

\begin{itemize}
\item[(1)]$f_*[X]=[P]$, i.e., $f$ has degree $1$,
\item[(2)]$f_* \colon \pi_1(X) \to \pi_1(P)$ is an isomorphism.
\end{itemize}
In this case, we shall say that $f$ realizes $X \succ P$.  Note that $f$ is not unique with respect to the properties $(1)$ and $(2)$,
i.e., there could exist a map $g\colon X \to P$ satisfying $(1)$ and  $(2)$, but not homotopic to $f$.

Let $X$, $X'$ and $P$ be $PD_4$-complexes such that $X\succ P$ and  $X'\succ P$ are realized by $f\colon X\to P$ and 
$f'\colon X'\to P$.  One of the main questions that we want to address in this paper is,  when are $X$ and $X'$ homotopy equivalent 
over $P$?  We show that $X$ and $X'$  are homotopy equivalent over $P$ if and only if there is an isometry between the second 
homology groups (Theorem \ref{main}). 

\begin{remark}
Throughout the paper $\pi$ will denote the fundamental group $\pi_1(X)$.  Also note that for a $PD_4$-complex $X$, the integral 
group ring $\La:=\bZ\pi$ has an involution defined on it.  Every right(left) $\La$-module can be considered as a left(right) 
$\La$-module with the conjugate structure given by this involution.  Throughout this paper the functors $\otimes_{\La}$ and   
$\Hom_{\La}$ are defined using this fact.
\end{remark}

Starting with a $PD_4$-complex $X$, we also define a minimal $PD_4$-complex $P$ for $X$, called $X$-minimal, which is minimal with 
respect to the order relation $\succ$  (see Definition \ref{minimal}).  Minimal  $PD_4$-complexes are also considered by 
Hillman(\cite{hillman 06,hillman 09,hillman 13}) with special emphasis on a particular  type of minimal $PD_4$-complex, 
called a strongly minimal $PD_4$-complex.  
Recall that for a $PD_4$-complex $X$, the radical of the intersection form $\lambda_X$, denoted by 
$\Rad(\lambda_{X})$, is  isomorphic to the module $H^2(\pi; \La)$. 

\begin{definition} 
A $PD_4$-complex $P$ is said to be strongly minimal if $$H_2(P; \bZ[\pi_1(P)])/ \Rad(\lambda_{P})=0.$$
\end{definition}

\begin{remark}
Obviously, if $P$ is strongly minimal and $X \succ P$, then $P$ is $X$-minimal.
\end{remark}

These two notions of minimality coincide whenever  the cohomological dimension of the fundamental group is less than or equal to $2$ 
(see for example \cite[Theorem 25]{hillman 13}).  All known examples of strongly minimal models are $PD_4$-complexes with such
fundamental groups (\cite{hillman 06,hillman 09,hillman 13}).  Therefore one might 
consider the following natural question:

\begin{problem}
Find examples of (strongly) minimal $PD_4$-complexes whose fundamental group has cohomological dimension greater than $2$.
\end{problem}

Hillman \cite{hillman 06,hillman 13} gives a homotopy classification for $PD_4$-complexes over the strongly minimal models 
subject to a $k$-invariant constraint.  He considers the same obstruction as in the proof of our main result Theorem \ref{main}.
However, we point out that our method in this paper is different: to see that the obstruction vanishes Hillman realizes 
it by a self-equivalence, whereas we use a map $A'$ to relate the obstruction to intersection forms and cap products.  We also remove the 
hypothesis on the $k$-invariant.
  
The outline of the paper is as follows:  In Section two we list some of the immediate properties of the order relation 
$\succ$.  In Section three, for a $PD_4$-complex $X$, we define $X$-minimal $PD_4$ complexes.  
We show that if $H_2(X; \La)$ is finitely generated, than such minimal complexes exist (Theorem \ref{existence}).  
Section four is about Postnikov decomposition of the map $f\colon X \to P$.  
In section five, we prove our main result: two $PD_4$-complexes $X$ and $X'$ over the same minimal complex $P$ 
are homotopy equivalent if and only if there is an isometry  $\Phi \colon H_2(X; \La) \to H_2(X'; \La) $ (Theorem \ref{main}).  
\vskip .1cm

\noindent{\bf {Acknowledgements.}} The authors thank the referees for the clarifications of essential points of the paper and for 
several suggestions which led to a simplification of the proof of Theorem \ref{main}.  
This research was supported by the Slovenian-Turkish grants BI-TR/12-14-001 and 111T667, 
and  Slovenian Research grants P1-0292-0101, J1-5435-0101, and J1-6721-0101. 


\section{Some Remarks and Preliminary Results}
In this section we will list some of the immediate properties of the above definition of the order relation $\succ$.

\begin{itemize}

\item[\textbf{(1)}] The relation $\succ$ is transitive and since $\id \colon X \to X$ realizes $X \succ X$ it is clear that $\succ$ is also 
reflexive.
\vskip .5cm

\item[\textbf{(2)}] The relation $\succ$ is symmetric in the sense of the following theorem:

\begin{theorem}
If $X\succ P$ and $P\succ X$ then $X$ is homotopy equivalent to $P$.
\end{theorem}

\begin{proof}
Let $f$ and $g$ realize $X \succ P$ and $P\succ X$, respectively.  Then $g\circ f$ and $f\circ g$ realize $X\succ X$ and $P\succ P$, 
respectively.  Then by \cite[Theorem 3, page 15]{hillman 94} $f\circ g$ and $g\circ f$ are homotopy equivalences, hence $f$ and $g$ 
are homotopy equivalences.
\end{proof}
\vskip .5cm

\item[\textbf{(3)}] If $f$ realizes $X\succ P$, then
$$
K_2(f, \La) := \Ker(f_*\colon H_2(X; \La) \to H_2(P; \La))
$$
is stably $\La$-free.  Here $\La=\bZ[\pi_1(P)]$ is the integral group ring.  Moreover, the restriction of the intersection form
$$
\lambda_X\colon  H_2(X; \La) \times  H_2(X; \La) \to \La
$$
to $K_2(f, \La)$ is non-singular.  Also note that the module $K_2(f, \La)$ is finitely generated.  See \cite[Lemmas 2.3, 2.6, 5.1]{wallbook} 
for these arguments.
\vskip .5cm

\item[\textbf{(4)}]  The converse of $(3)$ is also true, as witnessed by the following theorem:
\begin{theorem} \label{stably free}
Let $X$ be a $PD_4$-complex and $G\subset H_2(X; \La)$ a stably free $\La$-submodule such that $\lambda_X$ restricted 
to $G$ is non-singular.  Then there is a $PD_4$-complex $P$ such that $X\succ P$ is realized by $f\colon X\to P$ with 
$K_2(f, \La)=G$.
\end{theorem}

\begin{proof}(see \cite{hrs 05})
If $G$ is $\La$-free with $\La$-basis $e_1, \ldots , e_r \in G \subset H_2(X; \La)\cong \pi_2(X)$, then we can take
$$
\displaystyle
P = X \cup_{\varphi_i} \cup_1^r D^3,
$$
where  $\ [\varphi_i]=e_i$, for  $ i=1,\ldots, r$  and  we have the inclusion map $f \colon X \hookrightarrow P$ realizing 
$X \succ P$.

If $G$ is stably $\La$-free, i.e., $\displaystyle{G\oplus \La^{2a}}$ is free with basis $e_1, \ldots , e_q$, we can consider 
the $PD_4$-complex
$$
\displaystyle{Y=X\sharp(\sharp_1^a S^2\times S^2)}
$$
so that $G\oplus \La^{2a} \subset H_2(Y; \La)$.  Then we can construct
\begin{eqnarray*}
P= Y \cup_{ \varphi_i}  \cup_1^q  D^3,
\end{eqnarray*}
where  $\ [\varphi_i]=e_i$, for   $i=1,\ldots, q$, and $f\colon Y\hookrightarrow P$ is the inclusion as above, realizing $Y\succ P$.

We claim that $f$ is homotopic to a map $g\colon Y \to P$ such that $g$ factors over the collapsing map $c\colon Y\to X$, i.e.,
there exists a map $f'\colon X\to P$ such that $g=f'\circ c$.  Since $c$ realizes $Y\succ X$, $f'$ realizes $X\succ P$.

To see that the claim is true, we write $\displaystyle{T=\sharp_1^a S^2\times S^2}$ and hence $Y= X \sharp T$.  
The connected sum is formed by deleting a $4$-disc $\mathring{D}^4$ from $T$, letting 
$\mathring{T}= T \setminus \mathring{D}^4$, and attaching it to $X\setminus \{ \textrm{interior of the} \ 4- \textrm{cell} \}$ 
along $S^3$.  Note that, forming connected sums of $PD_4$-complexes can be done by using representations of 
$X=K\cup_{\varphi}D^4$ where $K$ is a $3$-complex  \cite [Lemma 2.9]{wallbook}.   By construction, $f|_{\mathring{T}}$ is 
homotopic to the constant map by a homotopy $h_t\colon \mathring{T} \to P$. Applying the homotopy extension property, 
$h_t$ can be extended to a homotopy $H_t\colon Y\to P$ with $H_0=f$, $H_1(\mathring{T})=\{ \pt \}$.  Let $H_1=g$, which factors over 
$Y/\mathring{T}=X$.
\end{proof}

\item[\textbf{(5)}]
Any degree $1$-map $f\colon X\to P$ defines, by Poincar\'{e} duality, a split short exact sequence
$$
\begin{matrix}
\xymatrix{0\ar[r]& K_2(f, \La)\ar[r] & H_2(X; \La)\ar[r]^{f_*}&H_2(P; \La)\ar[r] \ar@/^1.25pc/[l]^{s_f} & 0\\
&& H^2(X; \La)\ar[u]^{\cap [X]}_{\cong} & H^2(P; \La)\ar[l]^{f^*}\ar[u]^{\cong}_{\cap [P]}}
\end{matrix}
$$
such that $\im s_f$ and $K_2(f, \La)$ are orthogonal with respect to $\lambda_X$ (see \cite[Theorem 5.2]{wallbook}).

\item[\textbf{(6)}]
Assume that we are given $f'\colon X\to P'$ realizing $X\succ P'$ with $K_2(f', \La)=G$ which is stably free  and
$\lambda_X$ restricted to $G$ is non-singular.  The above construction (see the proof of Theorem \ref{stably free}) also  provides 
$X\succ P$ realized by $f\colon X\to P$ with $K_2(f, \La)=G$.  For this situation we shall need the following lemma:
\end{itemize}

\begin{lemma}
There is a homotopy equivalence $h\colon P\to P'$ such that the diagram
$$
\begin{matrix}
\xymatrix{ P\ar[rr]^h &  & P' \\ & X\ar[ul]^f \ar[ur]_{f'} &}
\end{matrix}
$$
commutes up to homotopy.
\end{lemma}

\begin{proof}
Assume first that $G$ is $\La$-free, with base $e_1, \ldots, e_r$.  Hence $P=X \cup_{ \varphi_i}  \cup_1^r  D^3$,
where  $[\varphi_i]=e_i$, for  $i=1,\ldots, r$.  Since $\xymatrix{S^2\ar[r]^{\varphi_i}&X\ar[r]^{f'}&P'}$ is null homotopic,
the map $f'\colon X\to P'$ extends to $h\colon P\to P'$.  Obviously, $\xymatrix{h_*\colon \pi_q(P)\ar[r]^{\cong}&\pi_q(P')}$ 
for $q=1, 2$.  But $h$ is of degree $1$, hence by duality we get 
$\xymatrix{h_*\colon H_*(P; \La)\ar[r]^{\cong}& H_*(P'; \La)}$, so $h$ is a homotopy equivalence by the 
Hurewicz-Whitehead theorem.

If $G$ is stably $\La$-free, i.e., $\displaystyle{G\oplus \La^{2a}}$ is free, we first stabilize $X\sharp(\sharp_1^a S^2\times S^2)$.  
Then as in the proof of Theorem \ref{stably free}, the map  $X\sharp(\sharp_1^a S^2\times S^2) \to P$ factors over $X$:
$$
\begin{matrix}
\xymatrix{X\sharp(\sharp_1^a S^2\times S^2) \ar[r] \ar[rrd]& X \ar[r] \ar@{-->}[rd]& P' \ar[d]^{\simeq} \\ & & P. }
\end{matrix}
$$
\end{proof}

The above facts can be summarized as follows:

\begin{corollary}
Given a $PD_4$-complex $X$, there is a bijective correspondence between the following sets:
\vskip .03cm
$$
\{ (G, \lambda_X|_G) |~G\subset H_2(X; \La) \ \textrm{stably free} \ \La-\textrm{module}, \lambda_X|_G \  \textrm{non-singular} \}
$$
$$
\updownarrow
$$
$$
\{P| P: PD_4-complex, X\succ P\}/ \textrm{homotopy equivalence .}
$$
\end{corollary}


\maketitle
\section{Minimal $PD_4$-complexes}

We start this section by fixing a $PD_4$-complex $X$.

\begin{definition} \label{minimal}
We say that the $PD_4$-complex $P$ is a minimal $PD_4$-complex for $X$ ( or $X$-minimal for short) if
\begin{itemize}
\item[(1)] We have $X\succ P$, and

\item[(2)] Whenever $P\succ Q$ for some $PD_4$-complex $Q$, then $P$ is homotopy equivalent to $Q$.
\end{itemize}
\end{definition}

The goal of this section is to show that $X$-minimal $PD_4$-complexes exist.  The following observation follows 
easily from the previous section.

\begin{lemma} \label{composition}
Let $X\succ P_1 \succ P_2$ be realized by $f_0\colon X\to P_1$ and $f_1\colon P_1 \to P_2$.  Then

\begin{itemize}
\item[(a)] $ K_2(f_1\circ f_0, \La)\cong  K_2(f_0, \La) \oplus K_2(f_1, \La) . $

\item[(b)] Let $s_{f_i}$ denote the splitting defined by $f_i$  for $i=0, 1$ (see Section $1$, Remark $(5)$  to recall
the definition of $s_{f_i}$).  Then we have
$$
s_{f_1 \circ f_0} = s_{f_0}\circ s_{f_1} .
$$
\end{itemize}
\end{lemma}

\begin{proof}
\begin{itemize}
\item[(a)]  The first assertion follows from the following isomorphisms
\begin{align*}
H_2(X; \La)&\cong K_2(f_1\circ f_0, \La)\oplus H_2(P_2; \La)\\
&\cong K_2(f_0, \La)\oplus K_2(f_1, \La)\oplus H_2(P_2; \La).
\end{align*}
\item[(b)] The second assertion follows from the degree $1$-property of the maps $f_0$ and  $f_1$, and also by well-known 
formulas of the cup-(respectively cap) products.
\end{itemize}
\end{proof}

Suppose we are given an infinite sequence of $PD_4$-complexes
$$
P_0=X\succ P_1\succ P_2\succ \ldots \succ P_i\succ P_{i+1}\succ \ldots
$$
which are realized by
$$
f_0\colon X\to P_1, \ f_i\colon P_i\to P_{i+1} \ i=1, 2, \ldots
$$

Let $Q$ be the direct limit of $\{P_i, f_i\}$, and let $f\colon X\to Q$ be the limit of the maps $f_i$.
Note that in general, we cannot assume that $Q$ is a $PD_4$-complex.  By Lemma \ref{composition} (a), we have
\begin{align*}
K_2(f, \La)& \cong K_2(f_0, \La)\oplus K_2(f_1, \La)\oplus \ldots \\
&= \bigoplus_0^\infty K_2(f_i, \La).
\end{align*}

\begin{lemma} \label{direct summand}
For $f$ as above, we have $K_2(f, \La)\subset H_2(X; \La)$ as a direct summand.
\end{lemma}

\begin{proof}
Compatibility of direct limits with homology and exact sequences gives the following exact sequence
\begin{eqnarray*}
\displaystyle
0\to K_2(f, \La) \to  H_2(X; \La)\to  \lim_{\to} H_2(P_i; \La )\to  0 \ .
\end{eqnarray*}

For the next argument it is convenient to write explicitly the following ladder:
$$
\begin{matrix}
\xymatrix{
H_2(X; \La)\ar[r]& H_2(P_1; \La)\ar[r]& H_2(P_2; \La)\ar[r]& \cdots \ar[r]& H_2(P_i; \La)\ar[r]& H_2(P_{i+1}; \La)\ar[r]& \cdots\\
H^2(X; \La)\ar[u]_{\cong}& H^2(P_1; \La)\ar[u]_{\cong} \ar[l]& H^2(P_2; \La)\ar[u]_{\cong} \ar[l]& \cdots \ar[l]&
H^2(P_i; \La)\ar[u]_{\cong} \ar[l]& H^2(P_{i+1}; \La)\ar[u]_{\cong} \ar[l]& \cdots \ar[l].}
\end{matrix}
$$
with the obvious maps and isomorphisms.    Property (b) of Lemma \ref{composition} gives inclusions
$$
H_2(X; \La)  \supseteq H_2(P_1; \La)    \supseteq H_2(P_2; \La)   \supseteq \cdots \supseteq H_2(P_i; \La)  \supseteq \cdots
$$
and
\begin{eqnarray*}
\displaystyle
\lim_{\to}H_2(P_i; \La)\cong \bigcap_0^{\infty} H_2(P_i; \La) .
\end{eqnarray*}
Moreover,
\begin{eqnarray*}
\displaystyle
s= \lim_{\to} s_i \colon \lim_{\to} H_2(P_i; \La ) \to H_2(X; \La)
\end{eqnarray*}
is a splitting of the above exact sequence.

Alternatively\footnote{We thank the referee for providing us with this argument.}, to obtain the above short exact sequence, one can use the fact that the homology 
group of the colimit is the colimit of the homology groups.  Then the universal property of the colimit of the homology groups yields the splitting $s$.
\end{proof}

\begin{remark}
In the proof above the direct limit is identified with the inverse limit
\begin{eqnarray*}
\lim_{\to} H_2(P_i; \La) =  \lim_{\gets}H^2(P_i; \La)
\end{eqnarray*}
which in general is not equal to
\begin{eqnarray*}
H^2(\lim_{\to}P_i; \La)=H^2(Q; \La) .
\end{eqnarray*}
\end{remark}

\begin{theorem} \label{existence}
If $H_2(X; \La)$ is a finitely generated $\La$-module, then there are $X$-minimal  $PD_4$-complexes.
\end{theorem}

\begin{proof}
Note first that $X$ itself can be $X$-minimal.  This occurs if the following set
$$
\{ (G, \lambda_X|_G) | \ G\subset H_2(X; \La) \ \textrm{stably free} \ \La-\textrm{module}, \lambda_X |_G \  \textrm{non-singular} \}
$$
contains only the trivial submodule $0$.  If the above set contains non-trivial submodules, then one can choose an arbitrary $G$ in it and construct the $PD_4$-complex 
$P_1$  with $X\succ P_1$ as in Theorem \ref{stably free}.  If $P_1$ is not $X$-minimal, then one takes an element $G_1$ from
$$
\{ (G, \lambda_{P_1}|_G) | \ G\subset H_2(P_1; \La) \ \textrm{stably free} \ \La-\textrm{module}, \lambda_{P_1} |_G \  \textrm{non-singular} \}
$$
giving $ X \succ P_1 \succ P_2$.  Continuing in this way, one obtains a sequence
$$
X \succ P_1 \succ P_2 \succ \ldots
$$
realized by
$$
f_0, f_1, f_2, \ldots
$$
By Lemma   \ref{direct summand} we have the following splitting of $H_2(X; \La)$
$$
\displaystyle{
H_2(X; \La) \cong \bigoplus_i K_2(f_i, \La) \oplus H_2(Q; \La) .}
$$
Because $H_2(X; \La)$ is finitely generated, the direct sum $\bigoplus_i K_2(f_i, \La) $ is a finite direct sum, hence the sequence
$$
X \succ P_1 \succ P_2 \succ \ldots
$$
is finite of type
$$
X \succ P_1 \succ P_2 \succ \ldots \succ P_k  .
$$
Hence $P_k$ is a $X$-minimal $PD_4$-complex.
\end{proof}

Note that the proof above indicates that in general presumably there might be more than one $X$-minimal $PD_4$-complexes.
One might consider the following question:

\begin{problem}
Give examples of  several $X$-minimal $PD_4$-complexes.
\end{problem}


\maketitle
\section{Postnikov Decomposition of $X \succ P$}

Let $X$, $P$ be $PD_4$-complexes such that $X \succ P$, realized by $f\colon X\to P$, which we may assume to be a 
fibration, and let $G=K_2(f, \La)$.  Then we have a decomposition \cite[pp. 141-142]{baues},
$$
\begin{matrix}
\xymatrix{ & E_3\ar[d]^p  \\ X \ar[ur] ^{f_3}\ar[r]^{f}& P }
\end{matrix}
$$
\vskip .1cm
\noindent
where $p\colon E_3\to P$ is a fibration with fiber $K(G, 2)$.  The above diagram satisfies the following:

\begin{itemize}
\item[(1)]  The map $p \colon E_3 \to P$ is $3$-coconnected, i.e.,
\begin{align*}
p_* \colon & \pi_q(E_3) \to \pi_q(P)  \ \textrm{is an isomorphism for} \  q > 3,   \\
p_* \colon & \pi_3(E_3) \to \pi_3(P) \ \textrm{is a monomorphism}.
\end{align*}

\item[(2)]  The map $f_3 \colon X \to E_3 $ is $3$-connected, i.e.,
\begin{align*}
(f_3)_* \colon & \pi_q(X) \to \pi_q(E_3)  \ \textrm{is an isomorphism for} \  q = 1, 2,   \\
(f_3)_* \colon & \pi_3(X) \to \pi_3(E_3) \ \textrm{is an epimorphism}.
\end{align*}
\end{itemize}

Taking mapping cylinders of $f_3$, $p$ and  $f$ we get the following inclusions $X \subset E_3 \subset P$,
and now properties $(1)$ and $(2)$ above become
\begin{itemize}
\item[(1\textprime)] $\pi_q(P, E_3)=0$ for $q \geq 4$.
\item[(2\textprime)] $\pi_q(E_3, X)=0$ for $q \leq 3$.
\end{itemize}
Hence, up to homotopy equivalence $E_3$ can be constructed from $X$ by attaching cells of dimension $ \geq 4$, so
$X^{(3)} = (E_3)^{(3)}$.  In fact, this is the way $E_3$ is constructed.  Moreover, we have $f |_{X^{(3)}} = p |_{(E_3)^{(3)}}$.

Now, note that $\xymatrix{E_3 \ar[r]^{p} & P}$ is a $K(G, 2)$ fibration which is not necessarily simple, that is  $\pi_1(P)\cong \pi$
does not have to act trivially on the homotopy group $G$ of the fiber.  We refer the reader to \cite{robinson} for the details of the
theory of non-simple fibrations.

There is a classifying space for  $K(G, 2)$-fibrations denoted by $\widehat{K}(G, 3)$ as described in \cite{robinson}.
Let $Q=K(\hepta G, 1)$ where $\hepta G$ is the group of isomorphisms of the Abelian group $G$.  The universal covering space
$\widetilde{Q}$ is contractible and $\hepta G$ acts freely on it.  Then
$$
\widehat{K}(G, 3) = (K(G, 3) \times \widetilde{Q})/ \hepta G .
$$
Here $K(G, 3)$ is interpreted as a topological(Abelian) group on which $\hepta G$ acts from the left.  There is a universal 
$K(G, 2)$-fibration over $\widehat{K}(G, 3)$ as described in \cite[Section 2]{robinson} which classifies $K(G, 2)$-fibrations.  
Hence there is a classifying map
$
\widehat{k_3}\colon P \to \widehat{K}(G, 3)
$
for $p \colon E_3 \to P$.  Moreover there is an obvious fibration
$$
\xymatrix{ K(G, 3)\ar[r] &  \widehat{K}(G, 3) \ar[r] ^(0.6)q & Q  ,}
$$
for which the null-element in $K(G, 3)$ gives a section $s \colon Q \to \widehat{K}(G, 3)$.

There is a $\pi_1(P)$ action on $G=K_2(f, \La)$ and hence there is a homomorphism
$$
\xymatrix{ \pi=\pi_1(X)\cong \pi_1(P) \ar[r] ^(0.66){\rho} &  \hepta G }
$$
inducing $B\rho \colon B\pi_1 \to Q$, such that
\eqncount
\begin{equation}   \label{principal G}
\begin{matrix}
\xymatrix{ P\ar[r]^(0.35){\widehat{k_3}} \ar[d]_\chi & \widehat{K}(G, 3) \ar[d]^q \\
B\pi_1\ar[r]^{B\rho} & Q}
\end{matrix}
\end{equation}
commutes.  Here $\chi$ classifies the universal covering $\widetilde{P} \to P $.

Suppose $X' \succ P$ is realized by $f' \colon X' \to P$.  Set $G'=K_2(f', \La)$ and
$Q'=K(\hepta G', 1)$.  We obtain a similar diagram
\eqncount
\begin{equation}  \label{principal G'}
\begin{matrix}
\xymatrix{ P\ar[r]^(0.35){\widehat{k}'_3} \ar[d]_\chi &  \widehat{K}(G', 3) \ar[d]^{q'} \\
B\pi_1\ar[r]^{B\rho'} & \ \ Q' \ .}
\end{matrix}
\end{equation}
\noindent
Recall the non-degenerate hermitian forms
$$
\lambda=\lambda_X |_G \colon G\times G \to \La  \quad  \textrm{and}  \quad  \lambda'=\lambda_{X'} |_{G'} \colon G'\times G' \to \La ,
$$
and let $p' \colon E'_3 \to P$ be the  $K(G', 2)$ fibration constructed from $X'$.

\begin{proposition} \label{isometry}
If $\Phi \colon G' \to G$ is an isometry, then $p' \colon E'_3 \to P$ and $p\colon E_3 \to P$ are fiber homotopy equivalent.
\end{proposition}

\begin{proof}
The isometry $\Phi$ induces the following equivalences
$$
\xymatrix{\hepta G' \ar[r]_{\cong}^(0.5)a & \hepta G }, \  \xymatrix{ Q' \ar[r]_{\simeq}^(0.5)b & Q } \ \textrm{and } \
\xymatrix{ \widehat{K}(G', 3) \ar[r]_{\simeq}^(0.5)c & \widehat{K}(G, 3) } .
$$
Note also that  the definition of isometry includes commutativity of the following diagram:
$$
\xymatrix{ & \pi_1(X) \ar[r]^{\rho} \ar[dl]_{f_*} & \hepta G \ar[dd] \\
\pi_1(P) & & \\
& \pi_1(X') \ar[r]^{\rho'} \ar[ul]^{f'_*}& \hepta G' \ .}
$$
All these maps induce maps between the diagrams  (\ref{principal G}) and (\ref{principal G'}) when ${\widehat{k_3}}$ and ${\widehat{k'_3}}$
are deleted.  Therefore, we have $b \circ q' \circ {\widehat{k'_3}} = q \circ {\widehat{k_3}}$ which can be seen from the diagram below.
$$
\begin{matrix}
\xymatrix{
& P  \ar[dl]_{\chi} \ar[d]^{\widehat{k'_3}} &  P \ar[dr]^{\chi} \ar[d]_{\widehat{k_3}}  &\\
B\pi_1(P) \ar[d]_{B\rho'} & \widehat{K}(G', 3) \ar[r]^{c} \ar[d]^{q'} & \widehat{K}(G, 3)  \ar[d]^{q} & B\pi_1(P)\ar[d] ^{B\rho} \\
Q' \ar[r]^{=} & Q'  \ar[r]^{b} & Q \ar[r]^{=} & Q }
\end{matrix}
$$
All subdiagrams commute by the commutativity of the diagrams (\ref{principal G}) and (\ref{principal G'}), by the hypothesis and construction 
of the maps $b$ and $c$.  Moreover, the following diagram 
$$
\begin{matrix}
\xymatrix{
P  \ar[r] ^{=} \ar[d]_{\chi} &  P\ar[d]^{\chi} \\
B\pi_1(P) \ar[d]_{B\rho'} &  B\pi_1(P) \ar[d]^{B\rho} \\
Q' \ar[r]_{b} & Q}
\end{matrix}
$$
is commutative by hypothesis.

Hence we have 
$$
q \circ c \circ {\widehat{k'_3}} = b \circ q' \circ {\widehat{k'_3}} = 
b \circ B\rho' \circ \chi = B\rho \circ \chi = q \circ {\widehat{k_3}} \ .
$$ 
Recall that we have the following fibration :   
$\xymatrix{ K(G, 3)\ar[r] &  \widehat{K}(G, 3) \ar[r] ^(0.6)q & Q } $.  
Using obstruction theory as in \cite[Chapter 4]{baues}, particularly  $4.29$ (with local coefficients), the only obstruction for
$c \circ {\widehat{k'_3}} - {\widehat{k_3}}$ to be homotopic to the constant map belongs to
$$
H^3(P; \pi_3(G, 3)) = H^3(P; G) \cong H_1(P; G).
$$
The group on the right is trivial as  $-\otimes_{\La} G$ is right exact.  
Hence   $c \circ \widehat{k'_3} $, $\widehat{k_3} \colon P \to \widehat{K}(G, 3)$ are homotopic maps.  The result follows, since
$c$ is a homotopy equivalence.
\end{proof}

\begin{example}
Suppose $G=\oplus^m_1\La$, then $[K(G, 2)]^{(2)} = \vee^m_1(\vee_{g\in\pi_1}S^2_g)$ and $\pi_1$ acts on
$\vee_{g\in \pi_1}S^2_g$ by permutation.  This is the case when $\pi_1$ is the free group on $l$ generators, and
in this case $P=\sharp_1^l S^1\times S^3$.
\end{example}

Next, we are going to show that $p_* \colon \pi_3(E_3) \to \pi_3(P)$ is an isomorphism.  For this consider the following diagram 
of Whitehead sequences:
$$
\xymatrix{ 0 \ar[r] & \Ga(\pi_2(X)) \ar[r] \ar[d]_{\cong} & \pi_3(X) \ar[d] \ar[r] & H_3(X; \La) \ar[r] & 0 \\
& \Ga(\pi_2(E_3)) \ar[r] \ar[d] & \pi_3(E_3) \ar[r] \ar[d] & H_3(E_3; \La) \ar[r] & 0 \\
0 \ar[r] & \Ga(\pi_2(P)) \ar[r] & \pi_3(P) \ar[r] & H_3(P; \La) \ar[r] & 0 \ .}
$$
By Poincare duality, we have $ \xymatrix{p_* \circ (f_3)_* = f_* \colon H_3(X; \La)  \ar[r]^(0.66){\cong} & H_3(P; \La)}$ an isomorphism, 
hence  the map $\xymatrix{H_3(E_3; \La) \ar[r]^{p_*} & H_3(P; \La)}$ is surjective.  Since $\Ga(\pi_2(X)) = \Ga(\pi_2(P) \oplus G ) \to 
\Ga(\pi_2(P))$ is surjective too, note that  $\xymatrix{\pi_3(E_3) \ar[r]^{p_*} &\pi_3(P) }$ is also surjective.  By Property $(1)$, it is also 
injective, hence it must be an isomorphism.


\maketitle
\section{Classification Relative Order}

Let $X$, $X'$ and $P$ be $PD_4$-complexes such that $X\succ P$ and  $X'\succ P$ are realized by $f\colon X\to P$ and $f'\colon X'\to P$.   
As before we set $G=K_2(f, \La)$ and $G'=K_2(f', \La)$.

The question we want to consider in this section is,  when are $X$ and $X'$ homotopy equivalent over $P$?  In other words, does there
exist a homotopy equivalence  $h\colon X \to X'$  such that $f'\circ h$ is homotopic to $f$?
$$
\begin{matrix}
\xymatrix{ X \ar[rr]^h \ar[dr]_f & & X' \ar[dl]^{f'} \\
& P  &}
\end{matrix}
$$

Suppose such an $h$ exists, then the following sequences
\eqncount
\begin{equation}  \label{necessary}
\begin{matrix}
\xymatrix{ 0 \ar[r] & K_2(f, \La) \ar[r] \ar[d]_{\Phi}  &  H_2(X; \La) \ar[r] \ar[d]_{\Phi} & H_2(P; \La)  \ar[r] \ar[d]_{=} & 0 \\
0 \ar[r] & K_2(f', \La) \ar[r]  &  H_2(X'; \La) \ar[r]  & H_2(P; \La)  \ar[r]  & 0 }
\end{matrix}
\end{equation}
are isomorphic, namely $\Phi := h_*$ is an isometry.  We are going to prove that this condition is also sufficient.

\begin{theorem} \label{main}
Suppose there is an isometry $\Phi \colon H_2(X; \La) \to H_2(X'; \La) $  satisfying the diagram (\ref{necessary}).  
Then there is a homotopy equivalence $h \colon X \to X'$ over $P$ inducing $\Phi$.
\end{theorem}

\begin{remark}
We should point out that our result gives a classification over the complex $P$, whereas Baues and Bleile \cite{baues and bleile} 
give classification result over $B\pi$ and Hillman \cite{hillman 06,hillman 13}  gives a classification result over the strongly minimal model.
\end{remark}

\begin{proof}
Since there is an isometry $\Phi \colon K_2(f; \La) \to K_2(f'; \La)$, we have a homotopy equivalence $g$ between the Postnikov 
systems
$$
\begin{matrix}
\xymatrix{E_3\ar[rr]^g \ar[dr]_p & &E'_3 \ar[dl]^{p'} \\
& P  &}
\end{matrix}
$$
We are going to denote $g |_{E_3^{(3)}}$ by  $\overline{h}$, i.e., $\overline{h} \colon E_3^{(3)} = X^{(3)} \to X'^{(3)} =E'^{(3)}_3$ 
such that $\overline{h} = g |_{E_3^{(3)}}$.
Let $X= X^{(3)} \cup_{\varphi} D^4$ and  $X'=X'^{(3)} \cup_{\varphi'} D'^4$, where  $\varphi \colon S^3 \to X^{(3)} $ and 
$\varphi' \colon S^3 \to X'^{(3)}$ are the attaching maps of the $4$-cells \cite[Theorem 2.4]{wall}.  For simplicity, we denote
$\varphi(S^3)=\partial D^4$ and $\varphi' (S^3)=\partial D'^4$.  The obstruction to extend $\overline{h}$ over $X$ belongs to
$$
H^4(X; \pi_3(X'))\cong H_0(X; \pi_3(X')) = \pi_3(X') \otimes_{\La} \bZ .
$$
This obstruction is given by $\overline{w}=w \otimes_{\La} 1= (\overline{h}_*[\partial D^4] - [\partial D'^4]) \otimes_{\La} 1 $.
We first consider  the difference $w:=\overline{h}_*[\partial D^4] - [\partial D'^4] \in \pi_3(X'^{(3)})$.

\begin{lemma}
The class $w \in \pi_3(X'^{(3)})$ maps to zero under the Hurewicz homomorphism $\pi_3(X'^{(3)}) \to H_3(X'^{(3)}; \La)$.
\end{lemma}

\begin{proof}
We write $P=P^{(3)}\cup_{\psi} D^4$ and we have the following isomorphisms
$$
\xymatrix{f_* \colon H_4(X, X^{(3)}; \La) \ar[r]^(0.55){\cong} & H_4(P, P^{(3)}; \La), \\
f'_* \colon H_4(X', X'^{(3)}; \La) \ar[r]^(0.55){\cong} & H_4(P, P^{(3)}; \La) }
$$
by the degree-$1$ property of $f$ and  $f'$, respectively.  Consider the diagram
$$
\xymatrix{\pi_4(X, X^{(3)})=H_4(X, X^{(3)}; \La)\ar[r] \ar[d]& H_3(X^{(3)}; \La) \ar[r] \ar[d]& H_3(X; \La) \ar[r] \ar[d]^{f_*}_{\cong}& 0\\
\pi_4(P, P^{(3)})=H_4(P, P^{(3)}; \La)\ar[r] & H_3(P^{(3)}; \La) \ar[r] & H_3(P; \La) \ar[r]& 0\\
\pi_4(X', X'^{(3)})=H_4(X', X'^{(3)}; \La)\ar[r] \ar[u]& H_3(X'^{(3)}; \La) \ar[r] \ar[u]& H_3(X'; \La) \ar[r] \ar[u]_{f'_*}^{\cong}& 0 \ .}
$$
The vertical maps are induced by $f$ and $f'$, respectively.  The rightmost and leftmost vertical maps are isomorphisms because of 
Poincar\'{e} duality and the degree-$1$ properties.   Hence we have the following isomorphisms
$$
\xymatrix{f_* \colon H_3(X^{(3)}; \La) \ar[r]^(0.55){\cong} & H_3(P^{(3)}; \La), \\
f'_* \colon H_3(X'^{(3)}; \La) \ar[r]^(0.55){\cong} & H_3(P^{(3)}; \La)  .}
$$
It follows that $f_* [\partial D^4] = f'_* [\partial D'^4]$.  Also the diagram below commutes.
$$
\begin{matrix}
\xymatrix{ H_3(X^{(3)}; \La) \ar[rr]^{\overline{h}_*} \ar[dr]^{f_*}_{\cong} &  &  H_3(X'^{(3)}; \La) \ar[dl]_{f'_*} ^{\cong} \\
& H_3(P^{(3)}; \La)  &}
\end{matrix}
$$
Hence $f'_* \circ \overline{h}_* [\partial D^4] = f_* [\partial D^4] = f' _* [\partial D'^4]$ which implies 
$\overline{h}_* [\partial D^4] - [\partial D'^4] = 0 \in H_3(X'^{(3)}; \La)$.
\end{proof}

Note that if $w \in \pi_3(X'^{(3)})$ is zero, then $\overline{h}$ extends to a map $h\colon X \to X'$ making the diagram 
$$
\begin{matrix}
\xymatrix{ X\ar[rr]^h \ar[dr]_f & & X' \ar[dl]^{f'} \\
 & P  &}
\end{matrix}
$$
commutative up to homotopy.  The above arguments show that the map $h$ is then of degree $1$, and it follows from this that 
$h$ is a homotopy equivalence.

From the Whitehead sequence \cite{whitehead 52}
$$
\xymatrix{0 \ar[r] & \Ga(\pi_2(X'^{(3)})) \ar[r] & \pi_3(X'^{(3)}) \ar[r] & H_3(X'^{(3)}; \La) \ar[r] & 0 , }
$$
it follows that $w\in \Ga(\pi_2(X'^{(3)})) = \Ga(\pi_2(X'))$ which has a decomposition
$$
\Ga(\pi_2(X'))\cong \Ga(G' \oplus \pi_2) \cong \Ga(\pi_2(P)) \oplus \pi_2(P) \otimes G' \oplus \Ga(G') .
$$
We write $w=w_1 + w_2 + w_3 $ according to this decomposition and 
since it suffices to show that $w\otimes_{\La} 1 \in \pi_3(X'^{(3)}) \otimes_{\La} \bZ$ is zero, we have to prove:
\begin{itemize}
\item[(1)] $\overline{w_1} = w_1 \otimes_{\La} 1 = 0 \in \Ga(\pi_2(P)) \otimes_{\La} \bZ$,
\item[(2)] $\overline{w_2} = w_2 \otimes_{\La} 1 = 0 \in G' \otimes_{\La} \pi_2(P) = G' \otimes_{\La} H_2(P; \La)$,
\item[(3)] $\overline{w_3} = w_3 \otimes_{\La} 1 = 0 \in \Ga(G') \otimes_{\La} \bZ $.
\end{itemize}

We are going to consider the above components  one by one.  

\begin{lemma}\label{image}
The element $\overline{w} \in  \pi_3(X'^{(3)}) \otimes_{\La} \bZ$ maps to zero under the map induced by $f'$,
$$
f'_* \colon \pi_3(X'^{(3)}) \otimes_{\La} \bZ \to \pi_3(P^{(3)})\otimes_{\La} \bZ .
$$
\end{lemma}
\begin{proof}
Recall that $\overline{w}= w\otimes_{\La}1= (\overline{h}_*[\partial D^4] - [\partial D'^4]) \otimes_{\La} 1$.  
Now, for every oriented $PD_4$-complex $Y$ with $Y= Y^{(3)}\cup_{\alpha} D^4$ the composite map $\overline{\partial}_Y$,
$$
\xymatrix{H_4(Y; \bZ)\ar@/^2pc/[rr]^{\overline{\partial}_Y} \ar[r]& H_4(Y, Y^{(3)})\otimes_{\La}\bZ=\pi_4(Y, Y^{(3)})\otimes_{\La}\bZ \ar[r]& 
\pi_3(Y^{(3)})\otimes_{\La}\bZ,}
$$
sends $[Y]$ to $[\alpha]\otimes_{\La}1$.  Consider the commutative diagram

\eqncount
\begin{equation}  \label{ladder}
\begin{matrix} 
\xymatrix{H_4(X; \bZ) \ar[r]^{f_*} \ar[d]_{\overline{\partial}_X} &  H_4(P; \bZ)  \ar[d] & H_4(X'; \bZ)  \ar[l]_{f'_*}\ar[d]^{\overline{\partial}_{X'}} \\
\pi_3(X^{(3)}) \otimes_{\La}\bZ \ar[r]_{f_*} & \pi_3(P^{(3)}) \otimes_{\La}\bZ & \pi_3(X'^{(3)}) \otimes_{\La}\bZ \ar[l]^{f'_*}  \ .}
\end{matrix}
\end{equation}
Since the diagram
$$
\xymatrix{X^{(3)} \ar[rr]^{\overline{h}} \ar[dr]_{f} & & X'^{(3)} \ar[dl]^{f'} \\
& P &}
$$
is homotopy commutative, we have $f'_* \circ \overline{h}_* = f_*$ also in the lower line of (\ref{ladder}).  The result then follows because we have 
$f_*[X]=[P]=f'_*[X']$.
\end{proof}

We have the following diagram of Whitehead sequences 
$$
\xymatrix{ & 0\ar[d] & 0\ar[d] & & \\
 & \Ga(G')\oplus \pi_2(P^{(3)})\otimes G' \ar[d]& \Omega \ar[d]& & \\
 0\ar[r] & \Ga(\pi_2(X^{(3)})) \ar[r] \ar[d]_{\Ga(f'_*)} & \pi_3(X'^{(3)}) \ar[d]^{f'*} \ar[r] & H_3(X'^{(3)}; \La) \ar[r] \ar[d]^{f'^*}_{\cong} & 0 \\ 
 & \Ga(\pi_2(P^{(3)})) \ar[r] \ar[d] & \pi_3(P^{(3)}) \ar[r] \ar[d] & H_3(P'^{(3)}; \La) \ar[r] & 0 \\
& 0 & 0 & & }
$$
Note that $\Ga(f'_*)$ is induced from the split surjective homomorphism 
$$
\xymatrix{\pi_2(X^{(3)})\cong H_2(X'; \La) \ar[r]^{f'_*} & H_2(P;\La)\cong \pi_2(P^{(3)}),}
$$
so $\Ga(f'_*)$ is split surjective, too.  Therefore $f'_*\colon \pi_3(X'^{(3)}) \to \pi_3(P'^{(3)})$ is surjective with
$$
\Kernel = \Omega \cong  \Ga(G')\oplus \pi_2(P^{(3)})\otimes G' .
$$
Now $-\otimes_{\La}\bZ$ is right exact, so we have exactness of
$$
\Omega\otimes_{\La}\bZ \to \pi_3(X'^{(3)})\otimes_{\La}\bZ \to \pi_3(P'^{(3)})\otimes_{\La}\bZ \to 0 .
$$ 
It follows from Lemma \ref{image} that the obstruction 
$$
(\overline{h}_*[\partial D^4] - [\partial D'^4]) \otimes_{\La} 1 \in 
\im(\Omega\otimes_{\La}\bZ \to \pi_3(X'^{(3)})\otimes_{\La}\bZ),
$$
that is, it comes from $\Ga(G')\otimes_{\La}\bZ \oplus \pi_2(P^{(3)})\otimes_{\La}G'$.  This immediately implies that the component $\overline{w_1} \in \Ga(\pi_2(P)) \otimes_{\La} \bZ$ should vanish.
\begin{corollary}
The component $\overline{w_1} = w_1 \otimes_{\La} 1 = 0$.
\end{corollary}

To further analyze the obstruction we will use a map $A'$ 
which relates our obstruction $\overline{w}$ to intersection forms and cap products.  For conveniency 
we shall give the details in the following remark.  

\begin{remark} \label{connection}
Let $Y=Y^{(3)} \cup_{\alpha} D^4$ be an oriented $PD_4$-complex with $\alpha \colon S^3 \to Y^{(3)} \in \pi_3(Y^{(3)})$.  
Given $\beta \colon S^3\to Y^{(3)} \in \pi_3(Y^{(3)})$, we denote the complexes $Y_{\beta}= Y^{(3)}\cup_{\beta}D^4$ and 
$Y_{\alpha + \beta} = Y^{(3)}\cup_{\alpha + \beta} D^4$ (in this notation $Y=Y_{\alpha}$).  Then we have the following:
\begin{itemize}
\item[(1)] If $\beta \in \Ga(\pi_2(Y^{(3)}))\subset \pi_3(Y^{(3)})$, then $H_4(Y_{\beta}; \bZ)\cong \bZ$ with generator $[Y_{\beta}]$ 
given by the top cell.
\item[(2)] If $\beta  \in \Ga(\pi_2(Y^{(3)}))$, this also implies that $H_4(Y_{\alpha+\beta}; \bZ)\cong \bZ$ with a generator given by 
the top cell (see \cite[Lemma 4.3]{chr 12}). 

We define the map
$$
\tilde{A}\colon \Ga(\pi_2(Y^{(3)})) \to \Hom_{\La}(H^2(Y^{(3)}; \La), H_2(Y^{(3)}; \La))
$$ 
by
$$
\xymatrix{
 & \beta \ar[r] & \cdot \cap [Y_{\beta}]\colon H^2(Y_{\beta}; \La) \ar[r] \ar@{=}[d]& H_2(Y_{\beta}; \La)\ar@{=}[d] \\
& &  H^2(Y^{(3)}; \La) & H_2(Y^{(3)}; \La) }
$$
\noindent
Observe that we also have 
$$
\xymatrix{
\cdot \cap [Y_{\alpha+\beta}]\colon H^2(Y_{\alpha+\beta}; \La) \ar[r] \ar@{=}[d]& H_2(Y_{\alpha+\beta}; \La)\ar@{=}[d] \\
H^2(Y^{(3)}; \La) & H_2(Y^{(3)}; \La) }
$$
\noindent
It is obvious that $\cdot \cap [Y_{\beta}] = \cdot \cap [Y_{\alpha+\beta}] - \cdot \cap[Y]$.
\item[(3)] $\tilde{A}$ induces a map $A\colon \Ga(\pi_2(Y^{(3)}))\otimes_{\La}\bZ \to \Hom_{\La}(H^2(Y^{(3)}; \La), H_2(Y^{(3)}; \La))$.  
The map $A$ can be seen to be the composite map 
$$
\xymatrix{
\Ga(\pi_2(Y^{(3)}))\otimes_{\La}\bZ \ar[r]& H_2(Y^{(3)}; \La)\otimes_{\La}H_2(Y^{(3)}; \La) \ar[r]& \Hom_{\La}(H^2(Y^{(3)}; \La), H_2(Y^{(3)}; \La)).}
$$
where the first map is induced from $\Ga(\pi_2)\to \pi_2\otimes \pi_2=H_2(Y^{(3)}; \La)\otimes H_2(Y^{(3)}; \La)$ and 
the second one maps $u\otimes_{\La} v \to \{ \xi \to (\xi\cap u)v\}$.
\end{itemize} 

Take now $X=X^{(3)}\cup_{\varphi} D^4$ giving 
$$
\displaystyle
A\colon \Ga(\pi_2(X^{(3)}))\otimes_{\La}\bZ \to \Hom_{\La}(H^2(X^{(3)}; \La), H_2(X^{(3)}; \La))
$$
and similarly for $X'=X'^{(3)}\cup_{\varphi'} D'^4$ one obtains the map $A'$
$$
A'\colon \Ga(\pi_2(X'^{(3)}))\otimes_{\La}\bZ \to \Hom_{\La}(H^2(X'^{(3)}; \La), H_2(X'^{(3)}; \La)).
$$

Recall that $\overline{h}\colon X^{(3)}\to X'^{(3)}$, and 
$w=\overline{h}_*[\partial D^4] - [\partial D'^4] \in \pi_3(X')$ gives our obstruction 
$\overline{w}=w \otimes_{\La} 1 \in \Ga(\pi_2(X'^{(3)})) \otimes_{\La} \bZ$.  We have 
$\overline{h}_*[\partial D^4] = [\partial D'^4]+w$, so by Remark \ref{connection} $(2)$
$$
\displaystyle
[X'^{(3)}\cup_{\overline{h}\varphi} D^4] \in H_4(X'^{(3)}\cup_{\overline{h}\varphi} D^4; \bZ) \cong \bZ
$$
is the canonical generator and $\overline{h}\colon X^{(3)} \to X'^{(3)}$ extends to the map
$$
\tilde{h}\colon X=X^{(3)}\cup_{\varphi} D^{4} \to X'^{(3)}\cup_{\overline{h}\varphi} D^4
$$
in the obvious way.  Hence $\tilde{h}_*[X]=[X'^{(3)}\cup_{\overline{h}\varphi} D^4]$.  

Now by Remark \ref{connection} $(2)$, we have
\begin{align*}
A'(\overline{w})=\cdot \cap [X'^{(3)}\cup_w D^4] & =\cdot \cap [X'^{(3)}\cup_{\overline{h}\varphi} D^4] - 
\cdot \cap [X'^{(3)}] \\
& = \overline{h}_* \circ (\cap [X]) \circ \overline{h}^* - \cdot \cap [X'].
\end{align*}
\end{remark}

We shall now prove that $\overline{w_2} = 0$, and $\overline{w_3} = 0$.  According to the splittings 
$$
H_2(X'^{(3)}; \La)= H_2(P^{(3)}; \La) \oplus G' \ \textrm{and} \ H^2(X'^{(3)}; \La)= H^2(P^{(3)}; \La) \oplus G'^*,
$$
the map $A'$ has components 
\begin{align*}
A'_2 &\colon G' \otimes_{\La} H_2(P^{(3)}) \to \Hom_{\La}(G'^*,  H_2(P^{(3)}; \La)) \\
A'_3 &\colon \Ga(G')\otimes_{\La}\bZ \to \Hom_{\La}(G'^*,  G').
\end{align*}
Note that both maps are injective because $G'$ is stably free.  

We consider first $A'_3(\overline{w_3})$.  By our hypothesis the restriction of 
$\overline{h}^*$ to $G'^*$ is equal to 
$\Phi^*$, similarly the map $\Phi$ is the restriction of 
$\overline{h}_*$.  Moreover, the restriction of 
the cap product map $\cdot \cap [X]$ to $G^*$ is equal to the 
inverse of the adjoint 
$$
\hat{\lambda}_X \colon H_2(X^{(3)}; \La) \to \Hom_{\La}(H_2(X^{(3)}; \La), \La)
$$ 
restricted to $G$, i.e., to $(\hat{\lambda}_{X {|_{G}}})^{-1}$.  Hence 
$$
A'_3(\overline{w_3})= \Phi (\hat{\lambda}_{X {|_{G}}})^{-1} \Phi^* - (\hat{\lambda}_{X' {|_{G'}}})^{-1} .
$$
But 
$$
\xymatrix{
G\ar[d]_{\Phi} \ar[r]^{\hat{\lambda}_{X {|_{G}}}} & G^* \\
G' \ar[r]^{\hat{\lambda}_{X' {|_{G'}}}} & G'^* \ar[u]_{\Phi^*}}
$$
commutes because $\Phi$ is an isometry.  This shows that $A'_3(\overline{w_3})=0$, 
hence $\overline{w_3}=0$.

We now come to 
$$
A'_2(\overline{w_2}) \colon G'^* \to H_2(P^{(3)}; \La).
$$
For this we identify 
$$
G'^*= \coker (H^2(P^{(3)}; \La) \to H^2(X'^{(3)}; \La)) .
$$
We have the following diagram 
\eqncount
\begin{equation} \label{final}
\begin{matrix}
\xymatrix{
 & & & H_2(P^{(3)}; \La) \\
H^2(X'^{(3)}; \La) \ar[r]^{\overline{h}^*} \ar[d] & H^2(X^{(3)}; \La) \ar[r]^{\cap [X]} \ar[d] & H_2(X^{(3)}; \La)\ar[r]^{\overline{h}_*} & H_2(X'^{(3)}; \La) \ar[u]_{f_*}  \\
G'^* \ar[r]^{\Phi^*} & G^* \ar[r] & G\ar[r]^{\Phi} \ar[u] & G' \ar[u] }
\end{matrix}
\end{equation}
Commutativity of the right square follows from the hypothesis (\ref{necessary}).  Commutativity of the left square is a consequence of 
$\overline{h} \colon X^{(3)} \to X'^{(3)} $ being a map over $P^{(3)}$.  Consider the composition $f_* \circ \overline{h}_* \circ (\cap [X]) \circ \overline{h}^*$.
Its restriction to $G'^*$ is the lower row of (\ref{final}) followed by 
$$
G' \to H_2(X'^{(3)}; \La) \to H_2(P^{(3)}; \La), 
$$
hence is zero.  Moreover, note also that the following
$$
\xymatrix{
H^2(X'^{(3)}; \La) \ar[r]^{\cdot\cap [X']} \ar[d] & H_2(X'^{(3)}; \La) \ar[r] & H_2(P^{(3)}; \La) \\
G'^* \ar[r] & G' \ar[u] & } 
$$
implies that $\cdot\cap [X']$ restricted to $G'^*$ has a vanishing component in $H_2(P^{(3)}; \La)$.  
Therefore $A'_2(\overline{w_2})=0$, implying $\overline{w_2}=0$.
\end{proof}
 

\bibliographystyle{amsplain}
\providecommand{\bysame}{\leavevmode\hbox
to3em{\hrulefill}\thinspace}
\providecommand{\MR}{\relax\ifhmode\unskip\space\fi MR }
\providecommand{\MRhref}[2]{%
  \href{http://www.ams.org/mathscinet-getitem?mr=#1}{#2}
} \providecommand{\href}[2]{#2}

\end{document}